\documentclass[11pt]{amsart}
\usepackage{amscd,verbatim,showlabels}
\begin{document}

\title
{Regular functions on covers of nilpotent coadjoint orbits}

 \author{Dan Barbasch}
        \address[D. Barbasch]{Dept. of Mathematics\\
                Cornell University\\Ithaca, NY 14850}
        \email{barbasch@@math.cornell.edu}

\date{\today}

\maketitle




%

\newcommand\vh{\vspace{0.2in}}

\newcommand\al{\alpha}
\newcommand\chal{\check\alpha}

\newcommand{\dia}[8]{{#1-#3-&#4&-#5-#6-#7-#8\\ &|& \\ &#2&}}

\newcommand\C[1]{{\mathcal #1 }}
\newcommand\ovl[1]{{\overline #1}}
\newcommand\fk[1]{{\mathfrak #1}}

\newcommand\CH{{\mathcal H}}
\newcommand\CI{{\mathcal I}}
\newcommand\CK{{\mathcal K}}
\newcommand\CL{{\mathcal L}}
\newcommand\CN{{\mathcal N}}
\newcommand\CO{{\mathcal O}}

\newcommand\CP{{\mathcal P}}
\newcommand\CR{{\mathcal R}}
\newcommand\CS{{\mathcal S}}
\newcommand\CU{{\mathcal U}}
\newcommand\CV{{\mathcal V}}
\newcommand\CW{{\mathcal W}}
\newcommand\CX{{\mathcal X}}
\newcommand\CY{{\mathcal Y}}
\newcommand\CZ{{\mathcal Z}}

\newcommand\bB{{\mathbb B}}
\newcommand\bC{{\mathbb C}}
\newcommand\bF{{\mathbb F}}
\newcommand\bH{{\mathbb H}}
\newcommand\bN{{\mathbb N}}
\newcommand\bR{{\mathbb R}}
\newcommand\bW{{\mathbb W}}
\newcommand\bZ{{\mathbb Z}}

\newcommand\ve{{\ ^\vee ~}}
\newcommand\one{{{1\!\!1}}}
\newcommand\zero{{{0\!\!0}}}
\newcommand\ovr{\overline}

\newcommand\la{{\lambda}}
\newcommand\ep{{\epsilon}}
\newcommand\sig{{\sigma}}
\newcommand\ome{{\omega}}

\newcommand\fc{{\mathfrak c}}
\newcommand\fg{{\mathfrak g}}
\newcommand\fh{{\mathfrak h}}
\newcommand\fl{{\mathfrak l}}
\newcommand\fm{{\mathfrak m}}
\newcommand\fn{{\mathfrak n}}
\newcommand\fo{{\mathfrak o}}
\newcommand\fp{{\mathfrak p}}
\newcommand\fq{{\mathfrak q}}
\newcommand\fs{{\mathfrak s}}
\newcommand\ft{{\mathfrak t}}
\newcommand\fu{{\mathfrak u}}
\newcommand\fz{{\mathfrak z}}

\newcommand\fC{{\mathfrak C}}
\newcommand\fO{{\mathfrak O}}

\newcommand\gc{{\fg_c}}
\newcommand\kc{{\fk_c}}
\newcommand\lc{{\fl_c}}
\newcommand\qc{{\fq_c}}
\newcommand\sce{{\fs_c}}
\newcommand\tc{{\ft_c}}
\newcommand\uc{{\fu_c}}

\newcommand\tie{{\tilde e}}
\newcommand\wti{\widetilde }

\newcommand\wO{\ovr{\wti{\CO}}}
\newcommand\wmuO{{\widehat{\mu_\CO}}}
\newcommand\COm{{\CO_\fm}}
\newcommand\tCOm{{\widetilde{\COm}}}
\newcommand\CSA{Cartan subalgebra ~}
\newcommand\CSG{Cartan subgroup ~}

\newcommand\phih{{\phi^{\fh}_f}}
\newcommand\ie{{\it i.e. ~}}
\newcommand\eg{{\it e.g. ~}}

\newtheorem{theorem}{Theorem}[subsection]
\newtheorem{corollary}[theorem]{Corollary}
\newtheorem{lemma}[theorem]{Lemma}
\newtheorem{proposition}[theorem]{Proposition}
\newtheorem{conjecture}[theorem]{Conjecture}
\newtheorem{definition}[theorem]{Definition}

\newcommand\Ad{{\operatorname{Ad}}}
\newcommand\ad{{\operatorname{ad}}}
\newcommand\tr{{\operatorname{tr}}}

\section{Notation and preliminary results}\label{5}

\vh
All of the techniques and most of the results in this section are well
known implicitly or explicitly, \cite{McG}, \cite{G}, \cite{KLT} and
references therein.
 
The structure sheaf of a variety $Z$ will be denoted by $\CS_Z.$ We
will abbreviate $R(Z)$ for $\Gamma(Z,\CS_Z).$ 

\vh
Typically $\CO$ will denote the orbit of a nilpotent element $e$ in a
semisimple Lie algebra $\fg.$  The orbit is isomorphic to $G/G(e).$
Its universal cover $\wti\CO$ is isomorphic to $G/G(e)_0.$ By one of
Chevalley's theorems there is a representation  $\tilde V$ and a
vector $\tilde e=(e,\tilde v)\in \fg\oplus \tilde V$ such that its
orbit under $G$ is the universal cover (in other words the stabilizer
of $\tilde v$ is $G(e)_0$.  Given any subgroup $G(e)_0\subset
H\subset G(e),$ there is a corresponding cover $\wti\CO_H$ which can be
realized in the same way as the orbit of an element
$e_H=(e,v_H)\in\fg\oplus V_H.$  

\vh
Let $\{e,h,f\}$ be a Lie triple associated to $e.$ Let $\fg_{\ge 2}$
be the sum of the eigenvectors of $\ad h$ with eigenvalue greater than
or equal to $2.$ Let $P_e$ be the parabolic subgroup determined by
$h,$ \ie the parabolic subgroup corresponding to the roots with
eigenvalue greater than or equal to zero for $\ad h.$ It is well known
that the natural map
\begin{equation}\label{5.1}
m_e : G\times_{P_e} \fg_{\ge 2} \longrightarrow \ovr \CO,\qquad
(g,X)\mapsto gXg^{-1}
\end{equation}
is birational and projective. The birationality follows from
\cite{BV}. The projective property is in \cite{McG}. Indeed let $P$ be any
parabolic subgroup, $\CP:=G/P$  and $\Sigma\subset \fp$ be any closed
$P$-invariant subspace. Then we can embed  $G\times_P \Sigma$ in $\CP\times\fg$ via
$(g,X)\mapsto (gPg^{-1},gXg^{-1}).$ The image is $\{(gPg^{-1},X)\ :\
g^{-1}Xg\in\Sigma\}.$ It is closed because $G/P$
is complete, and this is the $G$-orbit of $\{(P,X)\ :\ X\in\Sigma\}.$ 
Then the map $m(g,X):=gXg^{-1}$ is the composition of this embedding
with the projection on the second factor. 

\vh
Let $P=MN$ be an arbitrary parabolic subgroup and $\CO_\fm\subset \fm$
be a nilpotent orbit. A $G$-orbit $\CO$ is called {\it induced }
from $\CO_\fm$ (\cite{LS}), if
\begin{equation}\label{5.2}
\CO\cap [\CO_\fm + \fn]\quad \text{ is dense in }\quad \CO_\fm +\fn.
\end{equation}
Let $\Sigma:=\CO_\fm +\fn.$ There is a similar {\it moment map}
\begin{equation}\label{5.3}
m:G\times_P \Sigma \longrightarrow \CO,\qquad (g,X)\mapsto gXg^{-1},
\end{equation}
It is projective for the same reason as before, but it is not always
birational. Precisely, if $e\in\Sigma\cap\CO,$ then
the generic fiber of $m$ is isomorphic to $G(e)/P(e).$ 

We will write $\CZ$ for $G\times_P\Sigma$ where $\Sigma=\CO_{\fm}+\fn.$ 
In general, write $A_G(e):=G(e)/G(e)_0$ (we suppress the subscript $G$ if it is
clear from the context). Recall from \cite{LS} that $G(e)_0=P(e)_0,$
so that there is an inclusion  $A_P(e)\subset A_G(e).$ If $e_\fm\in
\CO_\fm$ then there is a surjection $A_P(e)\to A_M(e_\fm)$. Given a
character $\phi$ of $A_M(e_\fm),$ we will denote by the same letter
its inflation to $A_P(e).$ 

A related result is the following. Let $e_\fm\in\CO_\fm$ and
$\la\in\fm$ be such that $C(\la)=\fm.$ Let $e=e_\fm+\fn$ be a
representative for the induced nilpotent. Let $\psi$ be a character of
$A(\CO_\fm)$ and $\Psi$ be the induced representation to $A(\CO)$.  We
regard them both as characters of the centralizers of the
nilpotents. Choose a (noninvariant) inner product on $\fg.$
\begin{proposition}\label{5.5} Let $(\mu,V)$ be a representation of $G.$ Then
$$
[\mu:Ind_M^G[R(\CO_\fm)_\psi]]\le [\mu: R(\CO)_\Psi]
$$
where $R(\CO)_\Psi:=\sum_{\rho\in A(\CO)} [\rho:\Psi]R(\CO)_\rho.$
\end{proposition}
\begin{proof} For $n\in\bN,$ consider $\frac{1}{n}\la +e_\fm.$   
There is $p_n\in P$ such that $\la_n:=Ad(p_n)(\frac{1}{n}\la
+e_\fm)=\frac{1}{n}\la +e.$ For each $n,$ let $X_n^1,\dots, X_n^k$ be
an orthonormal basis of $\fz(\la_n),$ the centralizer in $\fg$ of
$\la_n.$ We can extract a  subsequence such that the $X_n^i$ all
converge to an orthonormal basis of $\fz(e).$ 
Now let $v_1^n,\dots , v_l^n$ be an orthonormal basis of the space of
fixed vectors of $\fz(\la_n)$ in $V.$ We can again extract a subsequence
such that the $v_n^j$ all converge to an orthonormal set of vectors in
$V.$ Because $v_n^j$ are invariant under the action of the $X_n^i,$
their limits are invariant under an orthonormal basis of $\fz(e).$ Using
Frobenius reciprocity, this
proves the claim for the connected components of the centralizers, \ie
the corresponding statement for $R(\wti\CO_\fm)$ and $R(\wti\CO).$ The claim
of the proposition follows by a minor modification of the argument.
\end{proof}

For the case of a Richardson nilpotent orbit, we can prove this type
of result in a more geometric fashion.
Let $P=MN$ be a parabolic subgroup with Lie algebra $\fk p=\fk m +\fk
n.$ Denote again by $\la\in\fg$  a semisimple  element whose centralizer is
$\fk m,$ and which is positive on the roots of $\fk n.$ Let $e\in \fk
n$ be a representative of the Richardson induced orbit from this
parabolic subalgebra, and denote its $G$ orbit by $\CO.$ As before,
there is a map
\begin{equation}
  \label{eq:1.1}
m:  G\times_P\fk n\longrightarrow \fg,
\end{equation}
with image $\ovl{\CO}.$ Let $\CO'$ be the inverse image of $\CO.$
Identify representations of $A_P(e)$ and $A_G(e)$ with representations
of $G(e)$ by making them trivial on $G(e)_0.$ 
\begin{proposition}
$$
[\mu: Ind_P^G[triv] ] = \sum_{\rho\in\widehat{ A_G(e)}}
[\rho|_{A_P(e)}:triv][\mu:R(\CO)_\rho]. 
$$
\end{proposition}
\begin{proof} 
When restricted to $\CO',$ the fiber of $m$ is $A_G(e)/A_P(e),$ and 
\begin{equation}
  \label{eq:1.2}
  R(\CO')\cong \sum [\rho|_{A_P(e)}:triv]R(\CO)_\rho.
\end{equation}
Let $\wti{\CO}$ be the cover corresponding to $A_P(e),$ and let $\C X$
be the normalization of $\ovr{\wti{\CO}}.$ Let $\wO_{reg}$ and $\C
X_{reg}$ be the regular points of the repective varieties. 

Then we have a diagram
\begin{equation}\label{5.6}
\begin{CD}
\C X_{\wti{\CO}} @>>> \C X_{reg} @>>>    \C X    \\
@VVV       @VVV            @V{\theta}VV     \\ 
\wti{\CO}   @>>> \wti{\CO_{reg}} @>>>  \ovr{\wti{\CO}}
\end{CD} 
\end{equation}
where $\CX_{\wti{\CO}}$ is the inverse image of $\wti{\CO}$ in $\C X.$
The codimensions of the complements of these sets is always greater
than or equal to 2, and the restriction of $\theta$ to $\C X_{reg}$ is an
isomorphism. This is because the morphism is \textit{finite}. 
Because $\C X$ is normal, we conclude that
\begin{equation}\label{5.7}
R(\wti{\CO})=R(\C X).
\end{equation}

Because $\CZ$ is smooth, it is also normal so there is a birational map
\begin{equation}\label{5.8}
\begin{CD}
\CZ @>\pi>> \C X 
\end{CD}
\end{equation}
This is also a \textit{finite} morphism. The rest of the proof is as in \cite{McG}.
\end{proof}
We will use this proposition in the setting of a triangular nilpotent
orbit, and the case (in the classical Lie algebras) where $A_P(e)=\{1\}.$

\bigskip
We return to the case where $P$ corresponds to the middle element
of the Lie triple. In this case, $A(\CO)\subset P.$
Let $\chi\in\widehat{A(\CO)}$ be a (1-dimensional) character viewed as a
representation of $G(e)$ trivial on $G(e)_0,$ and $\xi$ be a
representation of $P$ such that $\xi |_{G(e)}=\chi.$ Then 
\begin{equation}\label{5.9}
H^0(G/P,R(P\cdot f)\otimes \bC_\xi)\subset R(\CO,\CS_\chi)
\end{equation}
because $\CO$ embeds in $\CZ$ via $g\cdot e \mapsto [g,e].$ The results
in \cite{McG} imply that there is equality. Indeed, if $\phi\in
R(\CO,\CS_\chi),$ view it as a map $\phi:G\longrightarrow \bC$
satisfying
\begin{equation}
  \label{eq:5.10}
  \phi(gx)=\chi(x^{-1})\phi(g).
\end{equation}
 Then define a section $s_\phi\in H^0(G/P,R(P\cdot f)$ by the formula
 \begin{equation}
   \label{eq:5.11}
s_{\phi,\xi}(g)(p\cdot f):=\xi(p)\phi(gp).
 \end{equation}
The inverse map is given by 
\begin{equation}
  \label{eq:5.12}
s\mapsto  \phi_s(g):=s(g)(f).
\end{equation}
We note that there is another inclusion
\begin{equation}
  \label{eq:5.13}
  H^0(G/P,R(\fg_{\ge 2}\otimes \bC_\al)\subset H^0(G/P,R(P\cdot
  f)\otimes \bC_\xi).
\end{equation}
In \cite{McG} it is shown that when $\chi=triv$ and $\xi=triv,$ then
equality holds in (\ref{eq:5.13}), and in addition
\begin{equation}
  \label{eq:5.14}
  H^i(/P,R(g_{\ge 2})=(0) \quad\text{ for } i>0.
\end{equation}

We make the following conjecture
\begin{conjecture}\label{5.4} For each $\chi\in \widehat{A(\CO)}$ there
is a representation $\xi$ of $P_e$ satisfying $\xi\mid_{G^e}=\chi$ such that
\begin{equation*}
H^i(G/P_e,R(\fg_{\ge 2})\otimes \CS_\xi)=
\begin{cases}
R(\CO)_\chi, &\text{ if } i=0,\\
0 &\text{ otherwise. }
\end{cases}
\end{equation*}
\end{conjecture}

\vh 
A set of $\xi$ is given in the next section in the case of classical groups. 
The cases when $\CO$ is special and $A(\CO)=\ovl{A(\CO)}$ are called
\textit{smoothly cuspidal}.  View the complex
group $G$ as a real Lie group, and let $K$ be the maximal compact
subgroup.  Then $R(\CO)$ can be thought of as a $K$-module using the
identification of $K_c$ with $G.$  We will prove the following theorem
in the next section. 

\begin{theorem}\label{t:5.10}  Assume $\CO$ is smoothly cuspidal, and
  let $\chi\longleftrightarrow L_\chi$ be   the correspondence between
  characters of $\ovl{A(\CO)}$ and   unipotent representations defined
  in \cite{BV2}. Then 
  \begin{equation*}
    L_\chi\mid_{K}\cong R(\CO)_\chi.
  \end{equation*}
\end{theorem}
I conjecture that this result extends to the correspondence defined
in the next section for the classical Lie algebras, and that a
correspondence with these properties exists in the exceptional cases as well.

\vh
The purpose of \cite{McG} is to show that $R(\CO)$ can be expressed as
a combination of modules induced from characters on Levi
components. This has the effect one can express $R(\CO)$ as a
combination of restrictions to $K$ of standard modules. Theorem
\ref{t:5.10} sharpens this to say that in fact $R(\CO)_\chi$ equals the 
$K$-structure of an irreducible module in a natural way.

\vh
\section{The complex case}\label{6}

\vh
Given $\chi\in \widehat {A(\CO)},$ denote by $R(\CO)_\chi$ the regular
sections of the sheaf corresponding to $\chi.$  
\begin{conjecture}\label{6.2} Given a nilpotent orbit $\CO,$ there is
an infinitesimal character $\la_\CO$ with the following property.

There is a 1-1 correspondence $\chi\leftrightarrow X_\chi$ 
between characters of the component group and irreducible $(\fg,K)$
modules with WF-set $\ovr{\CO}$ and infinitesimal character $\la_\CO$
with the following properties:
\begin{enumerate}
\item The analogous character formulas as in \cite{BV2} hold,
\item $X_\chi |_K\cong R( \CO)_\chi,$
\item the $X_\chi$ are unitary.
\end{enumerate}
\end{conjecture}

As evidence for this conjecture we state the following theorem which is
the main result of this section. Recall Lusztig's quotient
$\ovr{A(\CO)}$ of the component group.

\begin{theorem}\label{6.2.1} The conjecture is true for classical
groups for nilpotent orbits such that $A(\CO)=\ovr{A(\CO)}.$
\end{theorem}

We call an orbit satisfying $A(\CO)=\overline{A(\CO)}$ {\textit stably trivial.}

\bigskip
We rely on \cite{BV2} and \cite{B1}. First, we prescribe the
infinitesimal character $\la_\CO.$  The main property will be that the
unipotent representations (irreducible $(\fg,K)$ modules whose
annihilator in the universal enveloping algebra is maximal with the
given infinitesimal character) are unitary and in 1-1 correspondence
with the irreducible characters of the component group $A(\CO).$ 
The notation is as in \cite{B1}. An orbit is called cuspidal if it is
not induced from any proper Levi component. For special orbits whose
dual is even, the infinitesimal character is one half the semisimple
element of the Lie triple corressponding to the dual orbit. For the
other orbits we need the case-by-case analysis.

\noindent{\bf Type A.}\quad A nilpotent orbit is determined by its
Jordan canonical form. It is given by a partition \ie a 
sequence of numbers in decreasing order $(n_1,\dots ,n_k)$ that add up
to $n.$  Let $(m_1,\dots ,m_l)$ be the dual partition. Then the
infinitesimal character is 
$$
(\frac{m_1-1}{2},\dots ,-\frac{m_1-1}{2},\dots
,\frac{m_l-1}{2},\dots, -\frac{m_l-1}{2})
$$
The orbit is induced from the trivial orbit on the Levi component
$GL(m_1)\times \dots \times GL(m_l).$ The corresponding
unipotent representation is spherical and induced irreducible from the
trivial representation on the same Levi component. All orbits are
{\it stably trivial}.

\noindent{\bf Type B.}\quad A nilpotent orbit is determined by its
Jordan canonical form (in the standard representation). It is
parametrized by a partition $(n_1,\dots ,n_k)$ of $2n+1$ such that
every even part occurs an even number of times. 
Let $(m'_0,\dots ,m'_{2p'})$ be the dual partition (add an $m'_{2p'}=0$ if
necessary, in order to have an odd number of terms).  
If there are any $m'_{2j}=m'_{2j+1}$ then pair them
together and remove them from the partition. 
Then relabel the remaining columns and pair them up, the rest of the
columns $(m_0)(m_1,m_2)\dots (m_{2p-1}m_{2p}).$ The members of each
pair have the same parity and $m_0$ is odd. Then form a parameter
$$
\aligned
(m_0)&\leftrightarrow (\frac{m_0-2}{2},\dots , 1/2),\\
(m'_{2j}=m'_{2j+1})&\leftrightarrow (\frac{m_{2j}-1}{2},\dots , -\frac{m_{2j}-1}{2})\\
(m_{2i-1}m_{2i})&\leftrightarrow (\frac{m_{2i-1}}{2},\dots , -\frac{m_{2i}-2}{2}) 
\endaligned$$
In case $m'_{2j}=m'_{2j+1},$ the nilpotent orbit is induced from a
parabolic subalgebra $\fp$ with Levi component $so(*)\times gl(m'_{2j})$ with
the trivial nilpotent on the $gl$ factor. The
component groups in $G$ and $P$ are equal. The unipotent representations are
unitarily induced irreducible from similar parameters on the Levi
component.  Similarly if some $m_{2i-1}=m_{2i},$ then the nilpotent is
induced irreducible from a nilpotent on an $so(*)\times
gl(\frac{m_{2i-1}+m_{2i}}{2})$ with the trivial nilpotent on the $gl$
factor. The component groups of the centralizers in $G$ and $P$ coincide. 
The unipotent representations are again induced irreducible from the
nilpotent orbit on $so(*)$ with partition the one for $\CO$ but with
$m_{2i-1},m_{2i}$ removed. The {\it stably trivial} orbits are the
ones such that every odd sized part appears an even number of
times except for the largest size. An orbit is trinagular if it has
partition $(1,1,3,3,\dots,2m-1,2m-1,2m+1).$ It is induced from the
trivial nilpotent orbit on $\fm=gl(2)\times gl(4)\dots gl(2m).$ The
component group $A_P$ is trivial.

\noindent{\bf Type C.}\quad A nilpotent orbit is determined by its
Jordan canonical form (in the standard representation). It is
parametrized by a partition $(n_1,\dots ,n_k)$ of $2n+1$ such that
every odd part occurs an even number of times. 
Let $(m'_0,\dots ,m'_{2p'})$ be the dual partition (add a $m'_{2p'}= 0$ if
necessary in order to have an odd number of terms). 
If there are any $m'_{2j-1}=m'_{2j}$ pair them up and
remove them from the partition. Then relabel and pair
up the remaining columns $(m_0m_1)\dots
(m_{2p-2}m_{2p-1})(m_{2p}).$ The members of each pair have the same
parity. The last one, $m_{2p},$ is always even. Then
form a parameter
$$
\aligned
(m'_{2j-1}=m'_{2j})&\leftrightarrow (\frac{m_{2j}-1}{2},\dots , 
-\frac{m_{2j}-1}{2})\\
(m_{2i}m_{2i+1})&\leftrightarrow (\frac{m_{2i}}{2},\dots , 
-\frac{m_{2i+1}-2}{2}),\\
m_{2p}&\leftrightarrow (\frac{m_{2p}}{2},\dots , 1).
\endaligned$$
The nilpotent orbits and the unipotent representations have the same
properties with respect to these pairs as the corresponding ones in
type B. The {\it stably trivial} orbits are the
ones such that every even sized part appears an even number of
times. An orbit is called triangular if it corresponds to the
partition $(2,2,4,4,\dots ,2m,2m).$ It is induced from the trivial
orbit on $\fm=sp(2m)\times gl(1)\times \dots\times gl(2m-1).$ The
component group $A_P$ is trivial.

\noindent{\bf Type D.}\quad A nilpotent orbit is determined by its
Jordan canonical form (in the standard representation). It is
parametrized by a partition $(n_1,\dots ,n_k)$ of $2n$ such that
every even part occurs an even number of times. 
Let $(m'_0,\dots ,m'_{2p'-1})$ be the dual partition (add a $m'_{2p-1}=0$ if
necessary). If there are any $m'_{2j}=m'_{2j+1}$ pair them up and
remove from the partition. 
Then pair up the remaining columns $(m_0m_{2p})(m_1,m_2)\dots
(m_{2p-2}m_{2p-1}).$ The members of each pair have the same parity and
$m_0, m_{2p-1}$ are even. Then form a parameter
$$
\aligned
(m'_{2j}=m'_{2j+1})&\leftrightarrow (\frac{m'_{2j}-1}{2}\dots , -\frac{m'_{2j}-1}{2})\\
(m_0m_{2p-1})&\leftrightarrow (\frac{m_0-2}{2},\dots ,-\frac{m_{2p-1}}{2}),\\
(m_{2i-1}m_{2i})&\leftrightarrow (\frac{m_{2i-1}}{2}\dots , -\frac{m_{2i}-2}{2}) 
\endaligned$$
The nilpotent orbits and the unipotent representations have the same
properties with respect to these pairs as the corresponding ones in
type B. An exception occurs when the partition is formed of pairs
$(m'_{2j}=m'_{2j+1})$ only. In this case there are two nilpotent
orbits corresponding to the partition. There are also two nonconjugate
Levi components of the form $gl(m'_0)\times gl(m'_2)\times \dots
gl(m'_{2p'-2})$ of parabolic subalgebras. There are two unipotent
representations each induced irreducible from the trivial
representation on the corresponding Levi component. The {\it stably
trivial} orbits are the ones such that every even sized part appears
an even number of times. A nilpotent orbit is triangular if it
corresponds to the partition $(1,1,3,3,\dots,2m-1,2m-1).$ It is
induced from the trivial orbit in the Levi component
$\fm=gl(2)\times\dots\times gl(2m-2).$ The component group $A_P$ is trivial.

\vh
Since all these results are clear for type A, we deal with types B,
C, D only. Consider a {\it stably trivial} nilpotent orbit
$\CO\subset\fg(n).$ Let $\fm=\fg(n)+gl(k_1)\times\dots\times gl(k_r)$
be a Levi component of a parabolic subalgebra in
$\fg^+:=\fg(n+k_1+\dots +k_r).$ There are $k_1,\dots k_r$ such that the orbit
\begin{equation}\label{6.3}
\CO^+ = Ind_{\fm}^{\fg(n')} [\CO\times triv\times \dots \times triv]
\end{equation} 
is {\it triangular.} Let $\fm^+$ be the Levi component corresponding
to the semisimple element of the Lie triple of $\CO^+.$ By
(\ref{5.5}), the unipotent representations attached to  $\CO^+$ 
\[\label{6.3.1}
X_\nu^+=R(\CO^+)_{\nu}- Y_\nu^+.
\]
where $Y_\nu^+$ is a genuine $K$-module.
Adding over $\nu$ and using \cite{BV1}, we find 
\[\label{6.3.2}
R(\wti\CO^+)=R(\wti\CO^+)-Y^+,
\]
so $X_\nu^+=R(\CO^+)_\nu.$ 

\medskip
By \cite{B3} and \cite{V}, for each unipotent representation $X_\psi$ 
there is a representation $\psi'$ of the centralizer of a
representative of the orbit $\CO$, and a $K$-representation $Y_\psi$
such that   
\begin{equation}\label{6.4}
X_\psi=R(\CO)_{\psi'}- Y_\psi.
\end{equation}  
In addition, $Y_\psi$ is supported on strictly
smaller orbits. We need to show that
$Y_\psi=0$ and $\psi=\psi'$. Consider the induced modules
\begin{equation}\label{6.5}
I_\psi^+=Ind_{\fm}^{\fg^+}[X_{\psi}\otimes triv].
\end{equation}
Let $\Psi$ be the character induced from $\psi$ and write $X_\Psi^+$ for
the corresponding combination of $X_\nu^+.$ By \cite{BV}
\begin{equation}\label{6.6}
I_{\psi}^+ =X_\Psi^+
\end{equation}
By (\ref{5.5}), the module induced from
$R(\CO)_{\psi'}$ is contained in $R(\CO^+)_{\Psi'}.$ Thus
\begin{equation}\label{6.7}
I_\psi^+=R(\CO^+)_{\Psi'} - Z_\psi^+,
\end{equation}
where $Z_\psi^+$ is a genuine module (containing the induced from
$Y_\psi$). Again by \cite{BV}, summing both sides over $\psi$ we get
\begin{equation}
\label{6.8}
R(\wti\CO^+)=\sum R(\CO^+)_{\Psi'} - Z^+
\end{equation}
where $Z^+=\sum Z_\psi^+$ is a genuine module supported on smaller
orbits. Unless $R(\wti\CO^+)=\sum R(\CO^+)_{\Psi'}$ and  $Z^+=0,$ this
contradicts the linear independence of the $R(\CO)_\nu$ in \cite{V}. Thus the
$Y_\psi=0$ and the correspondence $\psi\leftrightarrow\psi'$ is 1-1.

Remains to show that $\psi=\psi'.$ Suppose not. Equation  (\ref{6.7})
now reads $I_\psi^+=I_{\psi'}^+.$ 
The claim follows from the linear independence of the $R(\CO^+)_\nu$
as K-modules.

\ifx\undefined\bysame
\newcommand{\bysame}{\leavevmode\hbox to3em{\hrulefill}\,}
\fi


\begin{thebibliography}{10}

\bibitem[AB]{AB} 
J.~Adams, D.~Barbasch 
{\em The Reductive Dual Pairs Correspondence for Complex Groups}
J. of Func. An. vol 132, 1995


\bibitem[B1]{B1}
D.~Barbasch
{\em The unitary dual for complex classical Lie groups}
Invent. Math., vol 96, 1989, pp 103-176

\bibitem[B2]{B2}
{\em Unipotent representations for real reductive groups}
Proceedings of ICM 1990, Springer Verlag, Tokyo, 1991, pp. 769-777

\bibitem[B3]{B3}
{\em Orbital integrals of nilpotent orbits}
volume in honor of Harish-Chandra.

\bibitem[Br]{Br}
{\em Line bundles on the cotangent bundle of the flag variety}
Invent. Math., vol 113 Fasc. 1, 1993, pp. 1-20

\bibitem[BV]{BV}
N.~Berline, M.~Vergne
{\em Fourier transforms of orbits of the coadjoint representation}
Representation theory of reductive groups, Birkh\"auser-Boston,
Progress in mathematics vol 40, 1983, pp. 53-69


\bibitem[BV1]{BV1}
D.~Barbasch, D.~Vogan
{\em The local structure of characters} 
J. Funct. Analysis, vol. 34 no. 1, 1980, pp. 27-55


\bibitem[BV2]{BV2}
D.~Barbasch, D.~Vogan
{\em Unipotent representations of complex semisimple Lie groups} 
Ann. of Math., 1985, vol 121, pp. 41-110

\bibitem[BV3]{BV3}
D.~Barbasch, D.~Vogan
{\em Weyl group representations and nilpotent orbits} 
Representation theory of reductive groups, Birkh\"auser-Boston,
Progress in mathematics vol 40, 1983, pp. 21-35

\bibitem[G]{G}
W.~Graham, \emph{Functions on the universal cover of the principal
nilpotent orbit}, Inv. Math. \textbf{108} (1992), 15-27.

\bibitem[HC]{HC}
Harish-Chandra
{\em Fourier transform on a semisimple Lie algebra II }
Amer. J. of Math., vol 79, 1957, pp. 193-257

\bibitem[HC1]{HC1}
Harish-Chandra
{\em Harmonic analysis on semisimple Lie groups}
Bulletin AMS. vol 76, 1970, pp. 529-551

\bibitem[KP]{KP}
H-P.~Kraft, C.~Procesi
{\em On the geometry of conjugacy classes in classical groups}
Comm. Math. helv., 57, 1982, pp. 539-601

\bibitem[KV]{KV}
A.~Knapp, D.~Vogan
{\em Cohomological induction and unitary representations}
Princeton University Press, Princeton NJ, 1995

\bibitem[KLT]{KLT}
{\em Frobenius splitting of cotangent bundles of flag varieties}
Invent. Math., 136, 1999, pp. 603-621


\bibitem[LS]{LS}
G.~Lusztig, N.~Spaltenstein
{\em Induced unipotent classes}
J. London Math. Soc., vol. 19, 1979, pp. 41-52


\bibitem[McG]{McG}
M.~McGovern
{\em  Rings of regular functions on nilpotent orbits and their covers}
Invent. Math., 97, 1989, pp. 209-217

\bibitem[R]{R}
R.~Rao
{\em Orbital integrals in reductive groups}
Annals of Math., vol. 96 no. 3, 1972, pp. 505-510

\bibitem[Ro]{Ro}
W.~Rossmann
{\em Nilpotent orbital integrals in a real semisimple Lie algebra and
representations of Weyl groups}
Operator algebras, unitary representations, enveloping algebras and
invariant theory, Progress in Mathematics 92, Birkh\"auser Boston,
1990, pp. 263-287

\bibitem[SV1]{SV1}
W.~Schmid, K.~Vilonen
{\em Two geometric character formulas for reductive Lie groups }
J. Amer. math. Soc., vol. 11 no. 4, 1998, pp. 799-867

\bibitem[SV2]{SV2}
W.~Schmid, K.~Vilonen
{\em The B-V conjectures} preprint

\bibitem[V]{V}
D.~Vogan
{\em Associated varieties and unipotent representations}
Harmonic Analysis on reductive groups, Progress in Mathematics
vol. 101, Birkh\"auser, Boston-Basel-Berlin, 1991, pp. 315-388

\bibitem[W]{W}
N.~Wallach
{\em Invariant differential operators on a reductive Lie algebra and
Weyl group representations}
J. Amer. Math. Soc., vol. 6 no. 4, 1993, pp. 779-816
\end{thebibliography}
\end{document}